\documentclass[a4paper,10pt,leqno]{article}
\usepackage{amsmath, amssymb, amsthm}
\usepackage[mathscr]{eucal}

\newtheorem{thm}{\textbf{Theorem}}[section]
\newtheorem{prop}[thm]{\textbf{Proposition}}
\newtheorem{lem}[thm]{\textbf{Lemma}}
\newtheorem{cor}[thm]{\textbf{Corollary}}

\theoremstyle{definition}
\newtheorem{defn}[thm]{{\rm Definition}}
\newtheorem{rem}[thm]{{\rm Remark}}
\newtheorem{ex}[thm]{{\rm Example}}

\newcommand{\olapla}{\overline\bigtriangleup}
\newcommand{\onabla}{\overline\nabla}
\newcommand{\wnabla}{\widetilde\nabla}
\newcommand{\lapla}{\bigtriangleup}
\newcommand{\nablaf}{\nabla \hspace{-9pt} \nabla \hspace{-7.6pt} \nabla}
\newcommand{\wnablaf}{\widetilde\nabla \hspace{-9pt} \nabla \hspace{-7.6pt} \nabla}
\newcommand{\onablaf}{\overline\nabla \hspace{-9pt} \nabla \hspace{-7.6pt} \nabla}

\newcommand{\p}{\phi}

\begin{document}

\title{Geometry of $k$-harmonic maps and the second variational formula of the $k$-energy}

\author{\\
Shun Maeta
}

\date{}


\maketitle 


\begin{abstract}
In \cite{jell1}, J.Eells and L. Lemaire introduced $k$-harmonic maps,
 and T. Ichiyama, J. Inoguchi and H.Urakawa \cite{tijihu1} showed the first variation formula.
 In this paper, we give the second variation formula of $k$-harmonic maps,
 and show non-existence theorem of proper $k$-harmonic maps into a Riemannian manifold of 
 non-positive curvature $(k\geq 2)$. 
 We also study $k$-harmonic maps into the product Riemannian manifold,
 and describe the ordinary differential equations of $3$-harmonic curves and $4$-harmonic curves into a sphere,
 and show their non-trivial solutions.


\end{abstract}


\vspace{10pt}
\begin{flushleft}
{\large {\bf Introduction}}
\end{flushleft}
Theory of harmonic maps has been applied into various fields in differential geometry.
 The harmonic maps between two Riemannian manifolds are
 critical maps of the energy functional $E(\p)=\frac{1}{2}\int_M\|d\p\|^2v_g$, for smooth maps $\p:M\rightarrow N$.
 
On the other hand, in 1981, J. Eells and L. Lemaire \cite{jell1} proposed the problem to consider the {\em $k$-harmonic maps}:
 they are critical maps of the functional 
 \begin{align*}
 E_{k}(\p)=\int_Me_k(\p)v_g,\ \ (k=1,2,\dotsm),
 \end{align*}
 where $e_k(\p)=\frac{1}{2}\|(d+d^*)^k\p\|^2$ for smooth maps $\p:M\rightarrow N$.
G.Y. Jiang \cite{jg1} studied the first and second variation formulas of the bi-energy $E_2$, 
and critical maps of $E_2$ are called {\em biharmonic maps}. There have been extensive studies on biharmonic maps.
 
 Recently, in 2009, T. Ichiyama, J. Inoguchi and H. Urakawa \cite{tijihu1} studied the first variation formula of the 
$k$-energy $E_k$,
 whose critical maps are called $k$-harmonic maps.
 Harmonic maps are always $k$-harmonic maps by definition.
 In this paper, we study $k$-harmonic maps and show the second variational formula of $E_k$. 
 
In $\S \ref{preliminaries}$, we introduce notation and fundamental formulas of the tension field and 
 the $k$-stress energy tension field.

 In $\S \ref{k-harmonic}$, we recall $k$-harmonic maps, 
 and state the second variation formula of the $k$-energy $E_k$ which will be proved in $\S \ref{second}$.
 
In $\S \ref{nonposi}$, we show the non-existence theorem of proper $k$-harmonic maps into a Riemannian manifold of 
 non-positive curvature $(k\geq 2)$.

In $\S \ref{stable}$, we introduce the notion of stable $k$-harmonic maps, 
and show the non-existence theorem of non-trivial stable $k$-harmonic maps into constant sectional curvature manifolds.

 In $\S\ref{product}$, we show the reduction theorem of $k$-harmonic maps into the product spaces.

 Finally, in $\S\ref{sphere}$, we derive the necessary and sufficient condition to be $k$-harmonic maps into a sphere,
and determine the ODEs of the $3$-harmonic and $4$-harmonic curve equations into a sphere, 
and show their non-trivial solutions, respectively. 

We would like to express our gratitude to Professor Hajime Urakawa who introduced and helped to accomplish this paper.
And we also would like to express our thanks to Professor Jun-ichi Inoguchi who helped us during the period of our study.


\section{Preliminaries}\label{preliminaries}
Let $(M,g)$ be an $m$ dimensional Riemannian manifold,
 $(N,h)$ an $n$ dimensional one,
 and $\p:M\rightarrow N$, a smooth map.
 We use the following notation.
 The second fundamental form $B(\p)$
 of $\p$ is a covariant differentiation $\widetilde\nabla d\p$ of $1$-form $d\p$,
 which is a section of $\odot ^2T^*M\otimes \p^{-1}TN$.
For every $X,Y\in \Gamma (TM)$, let
 \begin{equation}
 \begin{split}
 B(X,Y)
&=(\widetilde\nabla d\p)(X,Y)=(\widetilde\nabla_X d\p)(Y)\\
&=\overline\nabla_Xd\p(Y)-d\p(\nabla_X Y)=\nabla^N_{d\p(X)}d\p(Y)-d\p(\nabla_XY). 
 \end{split}
 \end{equation}
 Here, $\nabla, \nabla^N, \overline \nabla, \widetilde \nabla$ are the induced connections on the bundles $TM$,
 $TN$, $\p^{-1}TN$ and $T^*M\otimes \p^{-1}TN$, respectively.
 
 If $M$ is compact,
 we consider critical maps of the energy functional
 \begin{align}
 E(\p)=\int_M e(\p) v_g,
 \end{align}
where $e(\p)=\frac{1}{2}\|d\p\|^2=\sum^m_{i=1}\frac{1}{2}\langle d\p(e_i),d\p(e_i)\rangle$
 which is called the {\em enegy density} of $\p$, and the inner product 
 $\langle \cdot ,\cdot \rangle$ is a Riemannian metric $h$. 
 The {\em tension \ field} $\tau(\p)$ of $\p$ is defined by
 \begin{align}
 \tau(\p)=\sum^{m}_{i=1}(\widetilde \nabla d\p)(e_i,e_i)=\sum^m_{i=1}(\widetilde \nabla _{e_i}d\p)(e_i).
 \end{align}
 Then, $\p$ is a {\em harmonic map} if $\tau(\p)=0$.
 
 The curvature tensor field $\widetilde R(\cdot, \cdot)$ of the Riemannian metric on the bundle 
 $T^*M\otimes \p^{-1}TN$ is defined as follows :
 \begin{align}
 \widetilde R(X,Y)
=\widetilde \nabla_X \widetilde \nabla_Y -\widetilde \nabla_Y\widetilde \nabla_X-\widetilde \nabla_{[X,Y]},
\ \ \ \ (X,Y\in \Gamma (TM)).
 \end{align}
 
 Furthermore, we define the following: For any $Z\in \Gamma(TM),$
 \begin{equation}
 \begin{split}
(\widetilde R(X,Y)d\p)(Z)
&=R^{\p^{-1}TN}(X,Y)d\p(Z)-d\p(R^M(X,Y)Z) \\
&=R^N(d\p(X),d\p(Y))d\p(Z)-d\p(R^M(X,Y)Z),
\end{split}
\end{equation}
where $R^M,R^N$ and $R^{\p^{-1}TN}$ are the Riemannian curvature tensor fields on 
$TM,TN$ and $\p^{-1}TN$, respectively.
 And we define
\begin{align}
\olapla
=\widetilde \nabla^* \widetilde \nabla
=-\sum^m_{k=1}(\widetilde \nabla_{e_k}\widetilde \nabla_{e_k}
-\widetilde \nabla_{\nabla_{e_k}e_k}), 
\end{align}
 is the {\em rough Laplacian}.

A section of $\odot ^2T^*M$ defined by $S_\p=e(\p)g-\p^*h$ is called the {\em stress-energy tensor field}, and 
$\p$ is said to satisfy the {\em conservation law} if ${\rm div}S_\p=0.$
 As in \cite{jell1}, 
 it holds that
$$({\rm div}S_\p)(X)=-\langle\tau(\p),d\p(X)\rangle,\ \ \  (X\in\Gamma(TM)),$$ 

We  define that $\p$ is said to satisfy the {\em $k$-conservation law} $(k\geq 2)$ ${\rm div}S_{\p}^k=0$, where the {\em $k$-stress energy tensor field} $S_{\p}^k$ 
is defined by 
\begin{align}
({\rm div} S^k_{\p})(X):=-\langle \olapla^{k-2}\tau(\p) , d\p(X)\rangle,\ \ (X\in \Gamma(TM)).
\end{align}

\section{$k$-harmonic maps}\label{k-harmonic}

J. Eells and L. Lemaire \cite{jell1} proposed the notation of $k$-harmonic maps. 
The Euler-Lagrange equations for the $k$-harmonic maps was shown by
T. Ichiyama, J. Inoguchi and H. Urakawa \cite{tijihu1}.
We first recall it briefly.
\begin{defn}[\cite{jell1}]
For $k=1,2,\dotsm$ the {\em $k$-energy functional}
 is defined by 
\begin{align*}
E_k(\phi)=\frac{1}{2}\int_M\|(d+d^* )^k\phi\|^2v_g,\ \ \phi\in C^{\infty}(M,N).
\end{align*}
Then, $\phi$ is {\em $k$-harmonic} if it is a critical point of $E_k,$ i.e., for all smooth variation $\{\phi_t\}$ of $\phi$ with $\phi_0=\phi$,
\begin{align*}
\left.\frac{d}{dt}\right|_{t=0}E_k(\phi_t)=0.
\end{align*}
We say for a $k$-harmonic map to be {\em proper} if it is not harmonic. 
\end{defn}


\vspace{10pt}

Then, the first variational formula of $E_k$ can be obtained as follows.
 First, notice the following lemma which will be used to show Theorem \ref{non posi harmonic}.
\begin{lem}[\cite{tijihu1}]
For $k=2,3,\dotsm,$ the $k$-energy functional $E_k$ is given as follows:

Case\ $1$:\ $k=2l$,\ $l=1,2,\dotsm$ $(k$ is even$)$.
$$E_{2l}(\phi)=\frac{1}{2}\int_M\|\underbrace{\overline\bigtriangleup\dotsm \overline\bigtriangleup}_{l-1}\tau(\phi)\|^2v_g.$$

Case\ $2$:\ $k=2l+1$,\ $l=1,2,\dotsm$ $(k$ is odd$)$.
$$E_{2l+1}(\phi)=\frac{1}{2}\int_M
\|\overline\nabla(\ \underbrace{\overline\bigtriangleup\dotsm \overline\bigtriangleup}_{l-1}\tau(\phi) )\|^2v_g.$$

\end{lem}

\vspace{10pt}

Then, we have
\begin{thm}[\cite{tijihu1}]\label{kharmonic}
Let $k=2,3,\dotsm.$ Then, we have
\begin{align*}
\left.\frac{d}{dt}\right|_{t=0}E_k(\phi_t)
=-\int_M\langle \tau_k(\phi),V \rangle v_g,
\end{align*}
where
\begin{align*}
\tau _k(\phi)
:=J\left(\overline \triangle ^{(k-2)}\tau (\phi )\right)
=\overline \triangle \left (\overline \triangle ^{(k-2)}\tau (\phi )\right)
-\mathscr{R} \left( \overline \triangle ^{(k-2)}\tau (\phi ) \right),
\end{align*}
and
\begin{align*}
\overline \triangle ^{(k-2)}\tau (\phi )
=\underbrace{\overline \triangle \dotsm \overline \triangle}_{k-2} 
\tau(\phi).
\end{align*}
\end{thm}

\vspace{10pt}
 As a corollary of this theorem, we have
\begin{cor}[\cite{tijihu1}]
$\phi :(M,g)\rightarrow (N,h)$ is a $k$-harmonic map if
\begin{equation}\label{propkharmonic}
\begin{split}
\tau _k(\phi)
:=J\left(\overline \triangle ^{(k-2)}\tau (\phi )\right)
=\overline \triangle \left (\overline \triangle ^{(k-2)}\tau (\phi )\right)
-\mathscr{R} \left( \overline \triangle ^{(k-2)}\tau (\phi ) \right)=0.
\end{split}
\end{equation}
\end{cor}

\vspace{10pt}


\vspace{10pt}
Notice here that any harmonic map is $k$-harmonic.

We recall the results of Jiang \cite{jg1} on the second variation formula of the $2$-energy $E_2$.

\begin{thm}[\cite{jg1}]
Let $\p:M\rightarrow N$ be a $2$-harmonic map from a compact Riemannian manifold 
 $M$ into an arbitrary Riemannian manifold $N$, 
 and $\{\p_t\}$ an arbitrary $C^{\infty}$ variation of $\p$ satisfying $(\ref{jg1 3.2})$,$(\ref{jg1 3.3})$.
Then, the second variation formula of $\frac{1}{2}E_2(\p_t)$ is given as follows.
\begin{equation}\label{jg1 second variation}
\begin{split}
\left.\frac{d^2}{dt^2}\right|_{t=0}E_2(\p_t)
&=\int_M\langle \overline\nabla^*\overline\nabla V+R^N(V,d\p(e_i))d\p(e_i),\\
&\hspace{50pt}-\overline\nabla^*\overline\nabla V +R^N(V,d\p(e_i))d\p(e_i)\rangle v_g\\
&+\int_M\langle V,(\nabla^N_{d\p(e_i)}R^N)(\tau(\p),d\p(e_i))V\\
&\ \ \ +(\nabla^N_{\tau(\p)}R^N)(V,d\p(e_i))d\p(e_i)\\
&\ \ \ +R^N(V,\tau(\p))\tau(\p)\\
&\ \ \ +2R^N(V,d\p(e_k))\overline\nabla_{e_k}\tau(\p)\\
&\ \ \ +2R^N(\tau(\p),d\p(e_i))\overline\nabla_{e_i}V\rangle v_g.
\end{split}
\end{equation}
\end{thm}

\vspace{10pt}

Then, we show the second variation formula of the $k$-energy $E_k$.

\begin{thm}\label{th second variation of k-harmonic}
Let $\p:M\rightarrow N$ be a $k$-harmonic map from a compact Riemannian manifold $M$ into an arbitrary
 Riemannian manifold $N$, and $\{\p_t\}$ an arbitrary $C^{\infty}$ variation of $\p$ satisfying $(\ref{jg1 3.2}),(\ref{jg1 3.3})$.
 Then, the second variation formula of $\frac{1}{2}E_{k}(\p_t)$ is given as follows.
\begin{equation}\label{second variation of k-harmonic}
\begin{split}
&\left.\frac{d^2}{dt^2}E_{k}({\p_t})\right|_{t=0}\\
&=\int_M
\left\langle
\onabla^*\onabla V-R^N(V,d\p(e_i))d\p(e_i),\right.\\
&\left.\hspace{100pt}
\olapla^{k-2}\{(\onabla^*\onabla V-R^N(V,d\p(e_i))d\p(e_i) )\}
\right\rangle v_g\\
&-\int_M
\left\langle
V,
\olapla ^{k-2}\{(\nabla^N_{d\p(e_k)}R^N)(d\p(e_k),V)\tau(\p)
+R^N(\tau(\p),V)\tau(\p)\right.\\
&\hspace{50pt}+R^N(d\p(e_k),\onabla_{e_k}V)\tau(\p)
+2R^N(d\p(e_k),V)\onabla_{e_k}\tau(\p)\}\\
&\hspace{50pt}
+(\nabla^N_V R^N)(d\p(e_i),\olapla^{k-2}\tau(\p))d\p(e_i) \\
&\hspace{50pt}
+R^N(\onabla_{e_i}V,d\p(e_i)) \olapla^{k-2}\tau(\p)\\
&\left.\hspace{50pt}
-2R^N(\olapla^{k-2}\tau(\p),d\p(e_i))\onabla_{e_i}V
\right\rangle v_g.
\end{split}
\end{equation}
\end{thm}

\section{Proof of the second variational formula of $k$-energy}\label{second}

In this section, we calculate the second variation formula the of $k$-energy $E_k$.

Assume that $\p:M\rightarrow N$ is a smooth map, $M$ is a compact Riemannian manifold, $N$ 
 and is a Riemannian manifold.
 First, let
\begin{align}\label{jg1 3.2}
\p_t:M\rightarrow N,\ t\in I_{\epsilon }=(-\epsilon ,\epsilon), \epsilon >0,	
\end{align}
be a $C^{\infty}$ one parameter variation of $\p$ which yields a vector field 
 $V\in \Gamma(\p^{-1}TN)$ along $\p$ in $N$ by
\begin{align}\label{jg1 3.3}
\p_0=\p,\ \left.\frac{\partial \p_t}{\partial t}\right|_{t=0}=V.
\end{align}
 The variation $\{\p_t\}$ yields a smooth map $F:M\times I_{\epsilon }\rightarrow N,$
 which is defined by 
\begin{align}
F(p,t)=\p_t(p),\ \ (p\in M,\  t \in I_{\epsilon}).
\end{align}
 
Taking the usual Euclidean metric on $I_{\epsilon}$, 
 and the product Riemannian metric on $M\times I_{\epsilon}$, 
 we denote by $\nablaf,\onablaf$ and $\wnablaf,$ the induced 
 Riemannian connection on $T(M\times I_{\epsilon}),
F^{-1}TN$ and $T^*(M\times I_{\epsilon})\otimes F^{-1}TN,$ respectively.
 If $\{e_i\}$ is an orthonormal frame field defined on a neighborhood  $U$ of $p\in M$,
 $\{e_i,\frac{\partial}{\partial t}\}$ is also an orthonormal frame field on a coordinate neighborhood 
 $U\times I_{\epsilon}$ in $M\times I_{\epsilon}$, and it holds that
 \begin{align}
 \nablaf_{\frac{\partial}{\partial t}}\frac{\partial}{\partial t}=0,\nablaf_{e_i}e_i=\nabla_{e_i}e_i,
 \nablaf_{\frac{\partial}{\partial t}}e_i=\nablaf_{e_i}\frac{\partial}{\partial t}=0.
 \end{align}
It also holds that 
\begin{equation}
\begin{split}
\frac{\partial f_t}{\partial t}=\frac{\partial F^{\alpha}}{\partial t}\frac{\partial}{\partial y^{\alpha}}
=dF(\frac{\partial}{\partial t}),
d\p_t(e_i)=dF(e_i),
\end{split}
\end{equation}
and
\begin{equation}
\begin{split}
(\wnabla_{e_i}d\p_t)(e_j)
&=\nabla^N_{d\p_t(e_i)}d\p_t(e_j)-df_t(\nabla_{e_i}e_j)
=(\wnablaf_{e_i}dF)(e_J)\\
(\wnabla_{e_k}\wnabla_{e_i}d\p_t)(e_j)
&=\nabla^N_{d\p_t(e_k)}((\wnabla_{e_i}d\p_t)(e_j))-(\wnabla_{e_i}d\p_t)(\nabla_{e_k}e_j)\\
&=(\wnablaf_{e_k}\wnablaf_{e_i}dF)(e_j),
\end{split}
\end{equation}
and so on. Then,


\begin{lem}[\cite{jg1}]\label{jg1 lem}
We have
\begin{equation}\label{jg1 3.13}
\begin{split}
&\hspace{-30pt}\onablaf_{\frac{\partial}{\partial t}}((\wnablaf_{e_i}dF)(e_i))\\
&=(\wnablaf_{e_i}\wnablaf_{e_i}dF)\left(\frac{\partial}{\partial t}\right)
-(\wnablaf_{\nabla_{e_i}e_i}dF)\left(\frac{\partial}{\partial t} \right)\\
&\hspace{70pt}
-R^N\left(dF(e_i),dF\left(\frac{\partial}{\partial t}\right)\right)dF(e_i),
\end{split}
\end{equation}

\begin{equation}\label{jg1 5.3}
\begin{split}
&\hspace{-30pt}\onablaf_{\frac{\partial}{\partial t}}\onablaf_{e_k}\onablaf_{e_k}((\wnablaf_{e_i}dF)(e_i))\\
&=\onablaf_{e_k}\onablaf_{e_k}
\left[
(\wnablaf_{e_i}\wnablaf_{e_i}dF)\left(\frac{\partial}{\partial t}\right)
-(\wnablaf_{\nabla_{e_i}e_i}dF)\left(\frac{\partial}{\partial t}\right)\right.\\
&\left.\hspace{120pt}
-R^N(dF(e_i),dF\left(\frac{\partial}{\partial t}\right)dF(e_i)
\right]\\
&\hspace{10pt}-\onablaf_{e_k}\left[
R^N(dF(e_k),dF\left(\frac{\partial}{\partial t}\right))
((\wnablaf_{e_i}dF)(e_i))\right]\\
&\hspace{10pt}-R^N(dF(e_k),dF\left(\frac{\partial}{\partial t}\right))\onablaf_{e_k}((\wnablaf_{e_i}dF)(e_i)),
\end{split}
\end{equation}
and
\begin{equation}\label{jg1 5.5}
\begin{split}
&\hspace{-30pt}\onablaf_{\frac{\partial}{\partial t}}\onablaf_{\nabla_{e_k}e_k}((\wnablaf_{e_i}dF)(e_i))\\
&=\onablaf_{\nabla_{e_k}e_k}
\left[
(\wnablaf_{e_i}\wnablaf_{e_i}dF)\left(\frac{\partial}{\partial t}\right)\right.
-(\wnablaf_{\nabla_{e_i}e_i}dF)\left(\frac{\partial}{\partial t}\right)\\
&\hspace{120pt}\left.
-R^N(dF(e_i),dF\left(\frac{\partial}{\partial t}\right) )dF(e_i)
\right]\\
&\hspace{10pt}-R^N(dF(\nabla_{e_k}e_k),dF\left(\frac{\partial}{\partial t}\right))((\wnablaf_{e_i}dF)(e_i)).
\end{split}
\end{equation}

\end{lem}



By using Lemma \ref{jg1 lem}, we prove Theorem \ref{th second variation of k-harmonic}.

\vspace{5pt}

\begin{flushleft}
Proof of Theorem \ref{th second variation of k-harmonic}.
\end{flushleft}
\vspace{-3pt}

As in \cite{tijihu1}, we have
\begin{equation}
\begin{split}
\frac{d}{dt}E_{k}(\p_t)
&=\int_M\langle dF(\frac{\partial}{\partial t}),
(\onablaf_{e_k}\onablaf_{e_k}
-\onablaf_{\nabla_{e_k}e_k)}
(\olapla^{k-2}((\wnablaf_{e_j}dF)(e_j)))\rangle v_g\\
&+\int_M\langle R^N (dF (\frac{\partial}{\partial t}),dF(e_i))dF(e_i),
\olapla^{k-2}((\wnablaf_{e_j}dF)(e_j))\rangle v_g.
\end{split}
\end{equation}
Therefore, we have
\begin{equation}
\begin{split}
\frac{d^2}{dt^2}E_k({f_t})
&=\int_M \langle \onablaf_{\frac{\partial}{\partial t}}dF(\frac{\partial}{\partial t}),
-\olapla(\olapla^{k-2}((\wnablaf_{e_i}dF)(e_i))) \\
&\hspace{60pt}+R^N (\olapla^{k-2}((\wnablaf_{e_i}dF)(e_i)),dF(e_j))dF(e_j)
\rangle v_g\\
&+\int_M \langle dF(\frac{\partial}{\partial t}),
\onablaf_{\frac{\partial}{\partial t}}[-\olapla(\olapla^{k-2}((\wnablaf_{e_j}dF)(e_j))) \\
&\hspace{60pt}+ R^N (\olapla^{k-2}((\wnablaf_{e_j}dF)(e_j)),dF(e_i))dF(e_i)]
\rangle v_g.
\end{split}
\end{equation}

\vspace{0.50pt}

Then, we calculate 
$\onablaf_{\frac{\partial}{\partial t}}\left[R^N(\olapla ^{k-2}((\wnablaf_{e_j}dF)(e_j)),dF(e_i))dF(e_i)\right].$
Using $(\ref{jg1 3.13})$ and second Bianchi's identity, we have
\begin{equation}\label{calculate of R^N}
\begin{split}
&\hspace{-30pt}\onablaf_{\frac{\partial}{\partial t}}\left[R^N(\olapla ^{k-2}((\wnablaf_{e_j}dF)(e_j)),dF(e_i))dF(e_i)\right]\\
&=(\nabla^N_{\olapla ^{k-2}((\wnabla_{e_j}dF)(e_j))}R^N)(dF(\frac{\partial}{\partial t}),dF(e_i))dF(e_i)\\
&+(\nabla^N_{dF(e_i)}R^N)(\olapla ^{k-2}((\wnablaf_{e_j}dF)(e_j)),dF(\frac{\partial}{\partial t}))dF(e_i)\\
&+R^N(\olapla^{k-2}\{(\wnablaf_{e_j}\wnablaf_{e_j}dF)\left(\frac{\partial}{\partial t}\right)
-(\wnablaf_{\nabla_{e_j}e_j}dF)\left(\frac{\partial}{\partial t} \right)\\
&\hspace{50pt}+R^N(dF(e_j),dF\left(\frac{\partial}{\partial t}\right))dF(e_j)\},dF(e_i))dF(e_i)\\
&+R^N(\olapla^{k-2} ((\wnablaf _{e_j}dF)(e_j)),(\wnablaf_{e_i}dF)(\frac{\partial}{\partial t}))dF(e_i)\\
&+R^N(\olapla^{k-2} ((\wnablaf_{e_j}dF)(e_j)),dF(e_i))(\wnablaf_{e_i}dF)(\frac{\partial}{\partial t}).
\end{split}
\end{equation}
Using $(\ref{jg1 3.13}),(\ref{jg1 5.3}), (\ref{jg1 5.5})\ and\ (\ref{calculate of R^N})$, we have

\begin{equation}\label{2diff of Ek}
\begin{split}
&\hspace{-30pt}\frac{d^2}{dt^2}E_{k}({f_t})\\
&\hspace{-10pt}
=\int_M \langle \onablaf_{\frac{\partial}{\partial t}}dF(\frac{\partial}{\partial t}),
-\olapla(\olapla^{k-2}((\wnablaf_{e_j}dF)(e_j))) \\
&\hspace{80pt}
+R^N (\olapla^{k-2}((\wnablaf_{e_j}dF)(e_j)),dF(e_i))dF(e_i)
\rangle v_g\\
&+\int_M
\left\langle
dF(\frac{\partial}{\partial t}),\right.\\
&\hspace{40pt}
\olapla^{k-2}
\left\{
\onablaf_{e_k}\onablaf_{e_k}
\left[
(\wnablaf_{e_j}\wnablaf_{e_j}dF)\left(\frac{\partial}{\partial t}\right)
-(\wnablaf_{\nabla_{e_j}e_j}dF)\left(\frac{\partial}{\partial t}\right)\right.\right.\\
&
\left.\hspace{160pt}
-R^N(dF(e_j),dF\left(\frac{\partial}{\partial t}\right))dF(e_j)
\right ]\\
&\hspace{40pt}-\onablaf_{e_k}\left[
R^N(dF(e_k),dF\left(\frac{\partial}{\partial t}\right))
((\wnablaf_{e_j}dF)(e_j))\right ]\\
&\hspace{40pt}
-R^N(dF(e_k),dF\left(\frac{\partial}{\partial t}\right))\onablaf_{e_k}((\wnablaf_{e_j}dF)(e_j))\\
&\hspace{40pt}
-\onablaf_{\nabla_{e_k}e_k}
\left[
(\wnablaf_{e_j}\wnablaf_{e_j}dF)\left(\frac{\partial}{\partial t}\right)
-(\wnablaf_{\nabla_{e_j}e_j}dF)\left(\frac{\partial}{\partial t}\right)\right.\\
&\left.\hspace{140pt}
-R^N(dF(e_j),dF\left(\frac{\partial}{\partial t}\right) )dF(e_j)
\right ]\\
&\hspace{40pt}
\left.-R^N(dF(\nabla_{e_k}e_k),dF\left(\frac{\partial}{\partial t}\right))((\wnablaf_{e_j}dF)(e_j))
\right\}\\
&\hspace{40pt}
+(\nabla^N_{\olapla ^{k-2}((\wnabla_{e_j}dF)(e_j))}R^N)(dF(\frac{\partial}{\partial t}),dF(e_i))dF(e_i)\\
&\hspace{40pt}
+(\nabla^N_{dF(e_i)}R^N)(\olapla ^{k-2}((\wnablaf_{e_j}dF)(e_j)),dF(\frac{\partial}{\partial t}))dF(e_i)\\
&\hspace{40pt}
+R^N(\olapla^{k-2}\{(\wnablaf_{e_j}\wnablaf_{e_j}dF)\left(\frac{\partial}{\partial t}\right)
-(\wnablaf_{\nabla_{e_j}e_j}dF)\left(\frac{\partial}{\partial t} \right)\\
&\hspace{50pt}-R^N(dF(e_j),dF\left(\frac{\partial}{\partial t}\right))dF(e_j)\},dF(e_i))dF(e_i)\\
&\hspace{40pt}
+R^N(\olapla^{k-2} ((\wnablaf _{e_j}dF)(e_j)),(\wnablaf_{e_i}dF)(\frac{\partial}{\partial t}))dF(e_i)\\
&\hspace{40pt}
\left.+R^N(\olapla^{k-2} ((\wnablaf_{e_j}dF)(e_j)),dF(e_i))(\wnablaf_{e_i}dF)(\frac{\partial}{\partial t})
\right\rangle v_g.
\end{split}
\end{equation}
Putting $t=0$ in $(\ref{2diff of Ek})$, the first term of the RHS of $(\ref{2diff of Ek})$ vanishes since $\p$ is $k$-harmonic.

Therefore, we have
\begin{equation}
\begin{split}
&\left.\frac{d^2}{dt^2}E_{k}({\p_t})\right|_{t=0}\\
&=\int_M
\left\langle
-\onabla^*\onabla V-R^N(d\p(e_i),V)d\p(e_i),\right.\\
&\left.\hspace{80pt}
\olapla^{k-2}\{(-\onabla^*\onabla V-R^N(d\p(e_j),V)d\p(e_j))\}
\right\rangle v_g\\
&-\int_M
\left\langle
V,
\olapla^{k-2}\{(\nabla^N_{d\p(e_k)}R^N)(d\p(e_k),V)\tau(\p)
+R^N(\tau(\p),V)\tau(\p)\right.\\
&\hspace{50pt}+R^N(d\p(e_k),\onabla_{e_k}V)\tau(\p)
+2R^N(d\p(e_k),V)\onabla_{e_k}\tau(\p)\}\\
&\hspace{50pt}
-(\nabla^N_{\olapla^{k-2}\tau(\p)}R^N)(V,d\p(e_i))d\p(e_i)\\
&\hspace{50pt}
-(\nabla^N_{d\p(e_i)}R^N)(\olapla^{k-2}\tau(\p),V)d\p(e_i)\\
&\hspace{50pt}
-R^N(\olapla^{k-2}\tau(\p),\onabla_{e_i}V)d\p(e_i)\\
&\hspace{50pt}
\left.-R^N(\olapla^{k-2}\tau(\p),d\p(e_i))\onabla_{e_i}V
\right\rangle v_g.
\end{split}
\end{equation}
By the Bianchi's second identity, 
\begin{equation}
\begin{split}
&-(\nabla^N_{\olapla^{k-2}\tau(\p)}R^N)(V,d\p(e_i))
-(\nabla^N_{d\p(e_i)}R^N)(\olapla^{k-2}\tau(\p),V)\\
&=(\nabla^N_{V}R^N)(d\p(e_i),\olapla^{k-2}\tau(\p)),
\end{split}
\end{equation}
and Bianchi's first identity
\begin{equation}
\begin{split}
&-R^N(\olapla^{k-2}\tau(\p),\onabla_{e_i}V)d\p(e_i)\\
&=R^N(\onabla_{e_i}V,d\p(e_i))\olapla^{k-2}\tau(\p)
-R^N(\olapla^{k-2}\tau(\p),d\p(e_i))\onabla_{e_i}V,
\end{split}
\end{equation}
we have Theorem \ref{th second variation of k-harmonic}.
$\qed$

\section{$k$-harmonic maps into Riemannian manifold of non positive sectional curvature}\label{nonposi}

Jiang \cite{jg1} showed the following proposition.


\begin{prop}[\cite{jg1}]\label{jg1 non posi harmonic}
Assume that $M$ is compact and $N$ is non positive curvature,
 $i.e.,$ Riemannian curvature of $N$ ${\rm Riem^N}$ $\leq 0.$
 Then, every $2$-harmonic map $\phi:M\rightarrow N$ is harmonic.
\end{prop}

\vspace{10pt}

In this section, we generalize this proposition to every $k$-harmonic map. Namely, we have

\vspace{10pt}

\begin{thm}\label{non posi harmonic}
Assume that $M$ is compact and $N$ is non positive curvature,
 ${\rm Riem^N}$ $\leq 0.$
 Then, every $k$-harmonic map $\phi:M\rightarrow N$ is harmonic.
\end{thm}

\vspace{10pt}
To prove this theorem, we show Theorem \ref{olapla harmonic}. First, we show the following two lemmas.


\begin{lem}\label{olapla1}
Let $l=1,2,\dotsm$. If \ $\onabla_{e_i}\olapla^{(l-1)} \tau(\p)=0\ (i=1,\dotsm,m)$, then $$\olapla^l \tau(\p)=0.$$
\end{lem}

\vspace{10pt}

\begin{proof}
Indeed, we can define a global vector field $X_{\p}\in \Gamma(TM)$ defined by
\begin{align}
X_{\phi}=\sum^m_{j=1}\langle-\overline\nabla_{e_j}\overline\bigtriangleup^{(l-1)}\tau(\phi),
\overline\bigtriangleup^l\tau(\phi)\rangle e_j.
\end{align}
 Then, the divergence of $X_{\p}$ is given as 
\begin{align*}
{\rm div(X_{\phi})}
&=\langle\overline\bigtriangleup^l\tau(\phi),\overline\bigtriangleup^l\tau(\phi)\rangle
+\sum^m_{j=1}\langle-\overline\nabla_{e_j}\overline\bigtriangleup^{(l-1)}\tau(\phi),
\overline\nabla_{e_j}\overline\bigtriangleup^{l}\tau(\phi)\rangle\\
&=\langle\overline\bigtriangleup^l\tau(\phi),\overline\bigtriangleup^l\tau(\phi)\rangle,
\end{align*}
by the assumption. Integrating this over $M$, we have
$$0=\int_M{\rm div}(X_{\phi})v_g
=\int_M\langle\overline\bigtriangleup^l\tau(\phi),\overline\bigtriangleup^l\tau(\phi)\rangle v_g,$$
which implies $\overline\bigtriangleup^l\tau(\phi)=0.$
\end{proof}

\vspace{10pt}

\begin{lem}\label{olapla2}
Let $l=1,2,\dotsm$. If\  $\olapla^l \tau(\p)=0$, then $$\onabla_{e_i}\olapla^{(l-1)}\tau(\p)=0,\ \ (i=1,\dotsm, m).$$
\end{lem}


\begin{proof}
Indeed, by computing the Laplacian of the $2l$-energy density $e_{2l}(\phi)$, we have
\begin{equation}\label{non.2}
\begin{split}
\bigtriangleup e_{2l}(\phi)
=&\sum^m_{i=1}\left\langle \overline \nabla_{e_i}\overline\bigtriangleup ^{(l-1)}\tau(\phi),
\overline \nabla_{e_i}\overline\bigtriangleup ^{(l-1)}\tau(\phi)\right\rangle\\
&\hspace{30pt}
-\left\langle \overline\nabla^*\overline \nabla(\overline\bigtriangleup ^{(l-1)}\tau(\phi)),\overline\bigtriangleup ^{(l-1)}\tau(\phi)\right\rangle \\
=&\sum^m_{i=1}\left\langle \overline \nabla_{e_i}\overline\bigtriangleup ^{(l-1)}\tau(\phi),
\overline \nabla_{e_i}\overline\bigtriangleup ^{(l-1)}\tau(\phi)\right\rangle 
\geq 0.
\end{split}
\end{equation}
By Green's theorem $\int_M\bigtriangleup e_{2l}(\phi) v_g=0,$ and $(\ref{non.2}),$ we have 
$\bigtriangleup e_{2l}(\phi)=0.$
 Again, by $(\ref{non.2}),$ we have
$$\overline \nabla_{e_i}\overline\bigtriangleup ^{(l-1)}\tau(\phi)=0,\ \ \ (i=1,\dotsm,m,\ \ \ l=1,2,\dotsm).$$
\end{proof}

\vspace{3pt}

\begin{lem}\label{olapla harmonic}
Let $l=1,2,\dotsm $. If $\olapla^l\tau(\p)=0$ or $\onabla_{e_i}\olapla^{(l-1)}\tau(\p)=0,\ 
\newline(i=1,2,\dotsm,m)$, then
every $k$-harmonic map $\p:M\rightarrow N$ from a compact Riemannian manifold $M$ into another Riemannian manifold $N$ is a harmonic map.
\end{lem}


\begin{proof}
By using Lemma \ref{olapla1}, \ref{olapla2}, we have Lemma \ref{olapla harmonic}.
\end{proof}

\vspace{10pt}

By using Lemma \ref{olapla harmonic}, we show Theorem \ref{non posi harmonic}.

\vspace{10pt}

\begin{flushleft}
Proof of Theorem \ref{non posi harmonic}.
\end{flushleft}

\vspace{-3pt}

By computing the Laplacian of the $2(k-1)$-energy density $e_{2(k-1)}(\phi),$ we have
\begin{equation}\label{laplae2(k-1)}
\begin{split}
&\hspace{-10pt}
\bigtriangleup e_{2(k-1)}(\phi)\\
&\hspace{-10pt}=\left\langle \overline \nabla_{e_i}\overline\bigtriangleup ^{(k-2)}\tau(\phi),
\overline \nabla_{e_i}\overline\bigtriangleup ^{(k-2)}\tau(\phi)\right\rangle
-\left\langle \overline\nabla^*\overline\nabla(\overline\bigtriangleup ^{(k-2)}\tau(\phi)),\overline\bigtriangleup ^{(k-2)}\tau(\phi)\right\rangle.
\end{split}
\end{equation}
By using 
\begin{align}\label{tauk0}
\tau_k (\phi)=\overline\bigtriangleup^{(k-1)}\tau(\phi)-\mathscr{R}(\overline\bigtriangleup^{(k-2)}\tau(\phi))=0.
\end{align}
Then, we have
\begin{equation}
\begin{split}
-&\left\langle \overline \nabla_{e_i}\overline\bigtriangleup ^{(k-2)}\tau(\phi),
\overline \nabla_{e_i}\overline\bigtriangleup ^{(k-2)}\tau(\phi)\right\rangle
=-\left\langle\mathscr{R}(\overline\bigtriangleup ^{(k-2)}\tau(\phi)),\overline\bigtriangleup ^{(k-2)}\tau(\phi)\right\rangle \\
&\geq 0,
\end{split}
\end{equation}
due to ${\rm Riem^N}\leq 0$.
By Green's theorem, 
\begin{equation}\label{green}
\begin{split}
0=\int_M\lapla e_{2(k-1)}(\phi)v_g
&=\int_M
\left\langle \overline \nabla_{e_i}\overline\bigtriangleup ^{(k-2)}\tau(\phi),
\overline \nabla_{e_i}\overline\bigtriangleup ^{(k-2)}\tau(\phi)\right\rangle v_g\\
&\hspace{20pt}
-\int_M
\left\langle\mathscr{R}(\overline\bigtriangleup ^{(k-2)}\tau(\phi)),\overline\bigtriangleup ^{(k-2)}\tau(\phi)\right\rangle v_g.
\end{split}
\end{equation}
Then, the both terms of the RHS of $(\ref{green})$ are non-negative, so that we have  
\begin{equation}\label{non.1}
\begin{split}
0&=\bigtriangleup e_{2(k-1)}(\phi)
=\left\langle \overline \nabla_{e_i}\overline\bigtriangleup ^{(k-2)}\tau(\phi),
\overline \nabla_{e_i}\overline\bigtriangleup ^{(k-2)}\tau(\phi)\right\rangle\\
&\hspace{100pt}-\left\langle\mathscr{R}(\overline\bigtriangleup ^{(k-2)}\tau(\phi)),\overline\bigtriangleup ^{(k-2)}\tau(\phi)\right\rangle.\\
\end{split}
\end{equation}
Then, since the both terms of $(\ref{non.1})$ are non-negative again, we have
\begin{align}
\overline\nabla_{e_i}\overline\bigtriangleup^{(k-2)}\tau(\phi)=0, \ \ \ \ (i=1,\dotsm,m).
\end{align}
By using Lemma \ref{olapla harmonic}, we have $\tau(\phi)=0.$
$\qed$

\section{Stable $k$-harmonic maps}\label{stable}
In this section, we generalize the results of Jiang \cite{jg1} on stable $2$-harmonic maps to stable $k$-harmonic maps.

By the second variation formula, Jiang \cite{jg1} defined the notion of stable $2$-harmonic maps as follows. 

\vspace{10pt}

\begin{defn}[\cite{jg1}]
Let $\p:M\rightarrow N$ be a $2$-harmonic map of a compact Riemannian manifold $M$ into
 a Riemannian manifold $N$. 
 Then, $\p$ is {\em stable} if the second variation of $2$-energy is non-negative for every 
 variation $\{\p_t\}$ along $\p$.
\end{defn}

\vspace{10pt}

Notice that by definition of the $2$-energy, any harmonic maps are stable $2$-harmonic maps.
 This also can be seen as follows:
 since $\tau(\p)=0$, for a vector field $V$ 
  of any variation $\{\p_t\}$, it holds that
\begin{align}
\left.\frac{d^2}{dt^2}E_2(\p_t)\right|_{t=0}
=\int_M\| -\overline\nabla^*\overline\nabla V+R^N(V,d\p(e_i))d\p(e_i)\|^2v_g\geq0.
\end{align}

\vspace{10pt}

\begin{thm}[\cite{jg1}]
Assume that $M$ is a compact Riemannian manifold,  and $N$ is a Riemannian manifold with a non-negative constant sectional curvature $K\geq 0$. Then, there is no non-trivial stable $2$-harmonic map satisfying 
the conservation law.
\end{thm}

\vspace{10pt}

By the second variation formula $({\rm cf.\ Theorem\ \ref{th second variation of k-harmonic}})$,
 we can introduce the notion of stable $k$-harmonic maps.

\vspace{10pt}

\begin{defn}
Let $\p:M\rightarrow N$ be any $k$-harmonic map of a compact Riemannian manifold $M$ into a Riemannian manifold $N.$ 
 Then, $\p$ is {\em stable} if the second variation of $k$-energy is non-negative for every variation $\{\p_t\}$ of $\p$, i.e., 
 $(\ref{second variation of k-harmonic})$ in Theorem \ref{th second variation of k-harmonic}
 is non-negative for every vector field $V$
 along $\p$.
\end{defn}

\vspace{10pt}
 
We have immediately
 \begin{prop}
  All harmonic maps $\p:M\rightarrow N$ are stable $2l$-harmonic maps
$(l=1,2,\dotsm)$.
\end{prop}

\vspace{5pt}

\begin{proof}
Let $k=2l\ (l=1,2,\dotsm).$
 Since $\tau(\p)=0$,
 by Theorem \ref{th second variation of k-harmonic},
 we have 
\begin{align*}
\left.\frac{d^2}{dt^2}E_{k}(\p_t)\right|_{t=0}
=\int_M\|\olapla^{l-1}(\onabla^*\onabla V-R^N(V,d\p(e_i))d\p(e_i))\|^2 v_g\geq 0.
\end{align*}
\end{proof}

Furthermore, one can consider a stable $k$-harmonic map into a Riemannian manifold $(N,h)$ 
 of constant sectional curvature. Then, we have

\vspace{10pt}

\begin{thm}\label{non stable k-harmonic}
Assume that $M$ is a compact Riemannian manifold,  and $N$ is a Riemannian manifold
 of non-negative constant sectional curvature $K\geq 0$. Then, there are no stable proper $k$-harmonic maps satisfying  
the conservation law, the $k$-conservation law and the $2(k-1)$-conservation law.
\end{thm}


\begin{proof}
Since $N$ has constant curvature, $\nabla^NR^N=0$, so that $(\ref{second variation of k-harmonic})$ becomes
\begin{equation}\label{second variation of k-harmonic 1}
\begin{split}
&\left.\frac{d^2}{dt^2}E_{k}({\p_t})\right|_{t=0}
=\int_M
\left\langle
\onabla^*\onabla V-R^N(V,d\p(e_i))d\p(e_i),\right.\\
&\hspace{100pt}
\left.
\olapla^{k-2}(\onabla^*\onabla V-R^N(V,d\p(e_i))d\p(e_i))
\right\rangle
v_g\\
&\hspace{70pt}-\int_M
\left\langle
V,
\olapla^{k-2}
\left\{R^N(\tau(\p),V)\tau(\p)\right.\right.\\
&\hspace{150pt}
+R^N(d\p(e_k),\onabla_{e_k}V)\tau(\p)\\
&\hspace{150pt}
\left.+2R^N(d\p(e_k),V)\onabla_{e_k}\tau(\p)\right\}\\
&\hspace{120pt}
+R^N(\onabla_{e_i}V,d\p(e_i)) \olapla^{k-2}\tau(\p)\\
&\hspace{120pt}
\left.-2R^N(\olapla^{k-2}\tau(\p),d\p(e_i))\onabla_{e_i}V
\right\rangle v_g.
\end{split}
\end{equation}
Especially, if we take $V=\olapla^{k-2}\tau(\p)$, then the first term of the RHS of
$(\ref{second variation of k-harmonic 1})$ must vanish. So we have

\begin{equation}\label{second variation of k-harmonic 2}
\begin{split}
&\hspace{-30pt}\left.\frac{d^2}{dt^2}E_{k}({\p_t})\right|_{t=0}\\
&\hspace{-10pt}=-\int_M
\langle
\olapla^{2(k-2)}\tau(\p),
R^N(\tau(\p),\olapla^{k-2}\tau(\p) )\tau(\p)\\
&\hspace{100pt}
+R^N(d\p(e_k),\onabla_{e_k}\olapla^{k-2}\tau(\p))\tau(\p)\\
&\hspace{100pt}
+2R^N(d\p(e_k),\olapla^{k-2}\tau(\p))\onabla_{e_k}\tau(\p)
\rangle v_g\\
&-\int_M
\langle
\olapla^{k-2}\tau(\p),
R^N(\onabla_{e_i}\olapla^{k-2}\tau(\p),d\p(e_i)) \olapla^{k-2}\tau(\p)\\
&\hspace{70pt}
-2R^N(\olapla^{k-2}\tau(\p),d\p(e_i))\onabla_{e_i}\olapla^{k-2}\tau(\p)
\rangle v_g.\\
&\hspace{-10pt}=-K\int_M
\langle
\tau(\p),\olapla^{2(k-2)}\tau(\p)
\rangle
\langle
\olapla^{k-2}\tau(\p),\tau(\p)
\rangle\\
&\hspace{30pt}
-\|\tau(\p)\|^2
\langle
\olapla^{k-2}\tau(\p),\olapla^{2(k-2)}\tau(\p)
\rangle\\
&\hspace{30pt}
+\langle
\onabla_{e_k}\olapla^{k-2}\tau(\p),\tau(\p)
\rangle
\langle
d\p(e_k),\olapla^{2(k-2)}\tau(\p)
\rangle\\
&\hspace{30pt}
-\langle
d\p(e_k),\tau(\p)
\rangle
\langle
\onabla_{e_k}\olapla^{k-2}\tau(\p),\olapla^{2(k-2)}\tau(\p)
\rangle\\
&\hspace{30pt}
+2\langle
\olapla^{k-2}\tau(\p),\onabla_{e_k}\tau(\p),
\rangle
\langle
d\p(e_k),\olapla^{2(k-2)}\tau(\p)
\rangle\\
&\hspace{30pt}
-2\langle
d\p(e_k),\onabla_{e_k}\tau(\p)
\rangle
\langle
\olapla^{k-2}\tau(\p),\olapla^{2(k-2)}\tau(\p)
\rangle\\
&\hspace{30pt}
+\langle
d\p(e_i),\olapla^{k-2}\tau(\p)
\rangle
\langle
\onabla_{e_i}\olapla^{k-2}\tau(\p),\olapla^{k-2}\tau(\p)
\rangle\\
&\hspace{30pt}
-\langle
\onabla_{e_i}\olapla^{k-2}\tau(\p),\olapla^{k-2}\tau(\p)
\rangle
\langle
d\p(e_i),\olapla^{k-2}\tau(\p)
\rangle\\
&\hspace{30pt}
-2\langle
d\p(e_i),\onabla_{e_i}\olapla^{k-2}\tau(\p)
\rangle
\|\olapla^{k-2}\tau(\p)\|^2\\
&\hspace{30pt}
+2\langle
\olapla^{k-2}\tau(\p),\onabla_{e_i}\olapla^{k-2}\tau(\p)
\rangle
\langle
d\p(e_i),\olapla^{k-2}\tau(\p)
\rangle
v_g.
\end{split}
\end{equation}
Here, we have
\begin{equation}
\begin{split}
\langle
d\p(e_i),\onabla_{e_i}\olapla^{k-2}\tau(\p)
\rangle
&=e_i
\langle
d\p(e_i),\olapla^{k-2}\tau(\p)
\rangle
-\langle
\onabla_{e_i}d\p(e_i),\olapla^{k-2}\tau(\p)
\rangle\\
&=-\langle
\tau(\p),\olapla^{k-2}\tau(\p)
\rangle
-\langle
d\p(\nabla_{e_i}e_i),\olapla^{k-2}\tau(\p)
\rangle\\
&=-\langle
\tau(\p),\olapla^{k-2}\tau(\p)
\rangle.
\end{split}
\end{equation}
By the assumptions, we have that 
\begin{equation}
\begin{cases}
&\langle\tau(\p),d\p(X) \rangle=0,\\
&\langle d\p(X),\olapla^{k-2}\tau(\p) \rangle=0,\\
&\langle d\p(X),\olapla^{2(k-2)}\tau(\p) \rangle=0.
\end{cases}
\end{equation}
 for all $X\in \Gamma(TM)$.
And we have $\langle d\p(e_k),\onabla_{e_k}\tau(\p) \rangle=-\|\tau(\p)\|^2$.
Thus, we have
\begin{equation}
\begin{split}
\left.\frac{d^2}{dt^2}E_{k}({\p_t})\right|_{t=0}
&=-K\int_M
\langle
\tau(\p),\olapla^{2(k-2)}\tau(\p)
\rangle
\langle
\olapla^{k-2}\tau(\p),\tau(\p)
\rangle\\
&\hspace{30pt}
+\|\tau(\p)\|^2
\langle
\olapla^{k-2}\tau(\p),\olapla^{2(k-2)}\tau(\p)
\rangle\\
&\hspace{30pt}
+2\langle
\tau(\p),\olapla^{k-2}\tau(\p)
\rangle
\|\olapla^{k-2}\tau(\p)\|^2
v_g.
\end{split}
\end{equation}

Now, we divide the situation into two cases.

Case $1)$\ \ $k=2l\ (l=1,2,\dotsm)$.
\ \ \ In this case, we have
\begin{equation}
\begin{split}
0\leq
\left.\frac{d^2}{dt^2}E_{2l}({\p_t})\right|_{t=0}
&=-K\int_M
\|\olapla^{2(l-1)}\tau(\p)\|^2
\|\olapla^{l-1}\tau(\p)\|^2\\
&\hspace{30pt}
+\|\tau(\p)\|^2
\|\olapla^{3(l-1)}\tau(\p)\|^2\\
&\hspace{30pt}
+2\|\olapla^{l-1}\tau(\p)\|^2
\|\olapla^{2(l-1)}\tau(\p)\|^2
v_g
\leq 0.
\end{split}
\end{equation}
By using Lemma \ref{olapla harmonic}, we obtain $\tau(\p)=0$.

\vspace{10pt}

Case $2)$\ \ $k=2l+1\ (l=1,2,\dotsm)$.
\ \ \ In this case, we have
\begin{equation}
\begin{split}
0\leq
\left.\frac{d^2}{dt^2}E_{2l+1}({\p_t})\right|_{t=0}
&=-K\int_M
\|\onabla_{e_i}\olapla^{2(l-1)}\tau(\p)\|^2
\|\onabla_{e_i}\olapla^{l-1}\tau(\p)\|^2\\
&\hspace{30pt}+
\|\tau(\p)\|^2
\|\onabla_{e_i}\olapla^{3l-2}\tau(\p)\|^2\\
&\hspace{30pt}
+2\|\onabla_{e_i}\olapla^{l-1}\tau(\p)\|^2
\|\olapla^{2l-1}\tau(\p)\|^2
v_g
\leq 0.
\end{split}
\end{equation}
By using Lemma \ref{olapla harmonic}, we obtain $\tau(\p)=0$.

So, we have Theorem \ref{non stable k-harmonic}.
\end{proof}


\section{The $k$-harmonic maps into the product spaces}\label{product}

In this section, we describe the necessary and sufficient condition to be 
$k$-harmonic maps into the product spaces. 
 First, let us recall the result of Y.-L. Ou \cite{ylo1}.
\begin{thm}[\cite{ylo1}]\label{thm1 of ylo1}
Let $\varphi :(M,g)\rightarrow (N_1,h_1)$ and 
$\psi :(M,g)\rightarrow (N_2,h_2)$ be two maps. Then, the map 
$\phi:(M,g)\rightarrow (N_1\times N_2,h_1\times h_2)$ with 
$\phi(x)=(\varphi(x),\psi(x))$ is $2$-harmonic if and only if the both map 
$\varphi$ or $\psi$ are $2$-harmonic. Furthermore, if one of $\varphi$ or $\psi$ is $2$-harmonic
and the other is a proper $2$-harmonic map, then $\phi$ is a proper $2$-harmonic map.
\end{thm}
We generalize Theorem \ref{thm1 of ylo1} for $k$-harmonic maps. Namely, we have
 the following theorem which is useful to construct examples the $k$-harmonic maps.
\begin{thm}\label{ps}
Let $\varphi :(M,g)\rightarrow (N_1,h_1)$ and 
$\psi :(M,g)\rightarrow (N_2,h_2)$ be two maps. Then, the map 
$\phi:(M,g)\rightarrow (N_1\times N_2,h_1\times h_2)$ with 
$\phi(x)=(\varphi(x),\psi(x))$ is $k$-harmonic if and only if the both map 
$\varphi$ or $\psi$ are $k$-harmonic. Furthermore, if one of $\varphi$ or $\psi$ is harmonic 
and the other is a proper $k$-harmonic map, then $\phi$ is a proper $k$-harmonic map.
\end{thm}

\begin{proof}
As in $\cite{ylo1}$, $\tau(\phi)=\tau(\varphi)+\tau(\psi),$
 $\overline \triangle_{\phi}\tau(\phi)=\overline \triangle_{\varphi}\tau(\varphi)+\overline \triangle_{\psi}\tau(\psi).$
 So we only notice that 
$\overline \triangle_{\varphi}\tau(\varphi)$ is tangent to $N_1$, $\overline \triangle_{\psi}\tau(\psi)$ is tangent to $N_2$.
So we have 
\begin{equation}
\begin{split}
\overline \triangle_{\phi} (\overline \triangle_{\phi}\tau(\phi))
&=\overline \triangle_{\phi}(\overline \triangle_{\varphi}\tau(\varphi)+\overline \triangle_{\psi}\tau(\psi))\\
&=\overline \triangle_{\varphi}(\overline \triangle_{\varphi}\tau(\varphi))
+\overline \triangle_{\psi}(\overline \triangle_{\psi}\tau(\psi)).
\end{split}
\end{equation}
Simillary,
\begin{align*}
\overline \triangle^{(k-2)}_{\phi}\tau(\phi)
=\overline \triangle^{(k-2)}_{\varphi}\tau(\varphi)
+\overline \triangle^{(k-2)}_{\psi}\tau(\psi).
\end{align*}
We use the property of the curvature of the product manifold to have
\begin{align*}
&R^{N_1\times N_2}(d\phi(e_i),\overline \triangle ^{(k-2)}_{\phi}\tau(\phi))d\phi(e_i)\\
=&R^{N_1}(d\varphi(e_i),\overline \triangle^{(k-2)}_{\varphi}\tau(\varphi))d\varphi(e_i)
+R^{N_2}(d\psi(e_i),\overline \triangle^{(k-2)}_{\psi}\tau(\psi))d\psi(e_i).
\end{align*}
Therefore, we have
\begin{equation}
\begin{split}
\tau _k(\phi)
&=\overline \triangle^{(k-1)}_{\phi}\tau(\phi)
+R^{N_1\times N_2}(d\phi(e_i),\overline \triangle ^{(k-2)}_{\phi}\tau(\phi))d\phi(e_i)\\
&=\overline \triangle^{(k-1)}_{\varphi}\tau(\varphi)
+R^{N_1}(d\varphi(e_i),\overline \triangle^{(k-2)}_{\varphi}\tau(\varphi))d\varphi(e_i)\\
&+\overline \triangle^{(k-1)}_{\psi}\tau(\psi)
+R^{N_2}(d\psi(e_i),\overline \triangle^{(k-2)}_{\psi}\tau(\psi))d\psi(e_i)\\
&=\tau_k(\varphi)+\tau_k(\psi).
\end{split}
\end{equation}
\end{proof}

\vspace{10pt}
The following corollary generalizes Corollary 3.4 in \cite{ylo1}.

\begin{cor}
Let $\psi:(M,g)\rightarrow (N,h)$ be a smooth map. Then, the graph 
$\p:(M,g)\rightarrow (M\times N,g\times h)$ with $\p(x)=(x,\psi(x))$ is a $k$-harmonic map 
 if and only if the map $\psi:(M,g)\rightarrow (N,h)$ is a $k$-harmonic map.
 Furthermore, if $\psi$ is proper $k$-harmonic, then so is the graph.  
\end{cor}

\begin{proof}
This follows from Theorem \ref{ps} with $\varphi:(M,g)\rightarrow (N,h)$ being 
 identity map which is harmonic.
\end{proof}

\vspace{10pt}

\begin{ex}
Let $\phi:\mathbb{R}\rightarrow \mathbb{R}\times S^n$ be a smooth curve parametrized by the 
 arc length in the product space $\mathbb{R}\times S^n$ with the standard
 product metric defined by 
 $$\p(x)=(x,\cos(\sqrt{2}x)c_1+\sin(\sqrt{2}x)c_2+c_4),$$
 where $c_1,\ c_2$ and $c_4$ are constant vectors in $\mathbb{R}^{n+1}$ 
 orthogonal to each other with $|c_1|^2=|c_2|^2=|c_4|^2=\frac{1}{2}$
 as in Example \ref{ex rcsmco1}. Then, $\p:\mathbb{R}\rightarrow \mathbb{R}\times S^n$ is $k$-harmonic for
 $k=2,3$ and $4$.
\end{ex}



\section{Determination of $k$-harmonic curves into a sphere}\label{sphere}

We determine that the ODEs of the $3$-harmonic, and $4$-harmonic curve equations into a sphere, respectively.


\begin{thm}\label{4}
Let $\gamma:I\rightarrow S^n\subset \mathbb{R}^{n+1}$ be a smooth curve defined on an interval of $\mathbb{R}$ parametrized by arc length.
 Then,  
$\gamma$ is $k$-harmonic curve if and only if
\begin{align*}
(\nabla _{\gamma '}\nabla _{\gamma '})^{(k-1)}&(\nabla_{\gamma'}\gamma')
+(\nabla _{\gamma '}\nabla _{\gamma '})^{(k-2)}(\nabla_{\gamma'}\gamma')\\
-&g\left(
(\nabla _{\gamma '}\nabla _{\gamma '})^{(k-2)}(\nabla_{\gamma'}\gamma'), \gamma' \right)
\gamma'=0,
\end{align*}
where $g$ is the standard Riemannian metric on $S^n$ of constant sectional curvature $1$,
we denote by $\gamma'$ the differential of $\gamma$ with respect to the arc length.
\end{thm}


\begin{proof}
$\gamma $ is a $k$-harmonic curve if and  only if
$$\tau _k(\gamma):=\overline \triangle \left (\overline \triangle ^{(k-2)}\tau (\gamma )\right)
-\mathscr{R} \left (\overline \triangle ^{(k-2)}\tau (\gamma)\right)=0.$$
Now, we have 
$\tau (\gamma )
=\nabla _{\gamma '}\gamma ',$
$\mathscr{R}(V)
=V-g\left (V,\gamma '\right)\gamma ',$
and 
 $\overline \triangle _\gamma V
=-\nabla _{\gamma '}\nabla _{\gamma '}V,$
 respectively for all
 $V\in \Gamma (\gamma ^{-1} TS^n).$
By 
$$\overline \triangle ^{(k-2)}V=(-1)^{(k-2)}(\nabla _{\gamma '}\nabla _{\gamma '})^{(k-2)}V,$$
 we have Theorem \ref{4}.
\end{proof}

\vspace{10pt}

\begin{prop}\label{5}
Let $\gamma :I\rightarrow S^n\subset \mathbb{R}^{n+1}$ be a smooth curve parametrized by arc length. 
Then $\gamma$ is $3$-harmonic curve if and only if 
\begin{equation}\label{3harmonic}
\begin{split}
\gamma ^{(6)}+2\gamma ^{(4)}
+(2-g_{22})\gamma ''
-4g_{23}\gamma '
+(2-3g_{22}-9g_{24}-8g_{33})\gamma =0,
\end{split}
\end{equation}
where $g_{ij}=g_0(\gamma ^{(i)},\gamma ^{(j)}),(i,j=0,1\dots),$
 and $g_0$ is the standard metric on the Euclidean space $\mathbb{R}^{n+1}$.
\end{prop}


\begin{proof}
$$\nabla ^0_{\gamma '} \gamma '=\sigma (\gamma ',\gamma ')+\nabla _{\gamma '} \gamma ',$$
which yields that
$$\nabla _{\gamma '} \gamma '=\nabla ^0_{\gamma '} \gamma '
+g(\gamma ',\gamma ')\gamma.$$
 Therefore, we have
$\nabla _{\gamma '}\gamma '
=\gamma ''+\gamma .$
 Similarly,
\begin{align}
(\nabla _{\gamma '}\nabla _{\gamma '})(\nabla _{\gamma '}\gamma ')\notag
&=\gamma ^{(4)}+\gamma ''+(g_{13}+1)\gamma .
\end{align}
\begin{equation}\label{nabla2}
\begin{split}
(\nabla _{\gamma '}\nabla _{\gamma '})^2&(\nabla _{\gamma '}\gamma ')\\
&=\gamma ^{(6)}+\gamma ^{(4)}+(g_{13}+1)\gamma ''
+(2g_{23}+3g_{14})\gamma '\\
&\hspace{30pt} +(1+g_{33}+3g_{24}+3g_{15}+g_{22}+3g_{13})\gamma.
\end{split}
\end{equation}
Here, we used
$g_{13}=-g_{22}\ ,\ g_{14}=-3g_{23}\ ,\ g_{15}=-3g_{33}-4g_{24}$.
So we have Proposition {\ref{5}}.
\end{proof}

\vspace{10pt}

\begin{prop}\label{6}
Let $\gamma:I\rightarrow S^n\subset \mathbb{R}^{n+1}$ be a smooth curve parametrized by the arc length. 
Then, $\gamma$ is $4$-harmonic curve if and only if
\begin{equation}\label{4harmonic}
\begin{split}
\gamma^{(8)}&+2\gamma^{(6)}+(2-g_{22})\gamma^{(4)}-11g_{23}\gamma^{(3)}\\
+&(-24g_{33}-25g_{24}+2-3g_{22})\gamma''\\
+&(19g_{34}+20g_{25}+9g_{16}-13g_{23})\gamma'\\
+&(5g_{44}+11g_{35}+10g_{26}+5g_{17}-40g_{33}-43g_{24}+g_{22}^2-5g_{22}+2)\gamma\\
=&0,
\end{split}
\end{equation}
where, $g_{ij}=g_0(\gamma ^{(i)},\gamma ^{(j)}),(i,j=0,1\dots).$
\end{prop}


\begin{proof}
We calculate $(\nabla_{\gamma'}\nabla_{\gamma'})^3(\nabla_{\gamma'}\gamma')$ as follows. 
\begin{align*}
&(\nabla_{\gamma'}\nabla_{\gamma'})^3(\nabla_{\gamma'}\gamma')\\
&=\gamma^{(8)}+\gamma^{(6)}+(g_{13}+1)\gamma^{(4)}+(4g_{23}+5g_{14})\gamma^{(3)}\\
&+(6g_{33}+15g_{24}+10g_{15}+1+g_{22}+3g_{13})\gamma''\\
&+(19g_{34}+20g_{25}+10g_{16}+12g_{23}+10g_{14})\gamma'\\
&+(5g_{44}+11g_{35}+10g_{26}+5g_{17}+10g_{33}+22g_{24}+14g_{15}+g^2_{13}+4g_{13}+1+g_{22})\gamma.
\end{align*}
By using $(\ref{nabla2})$, we have Proposition \ref{6}.
\end{proof}



\vspace{10pt}

We can derive the ODE to be $3$-harmonic or $4$-harmonic in terms of the Frenet-frame.
 Indeed, the Frenet-frame is given as follows:
\begin{equation}
\begin{cases}
&\gamma '=T\ ,\ \nabla^N _{\gamma '}T=\kappa N\ ,\ \nabla^N_{\gamma '}N=-\kappa T,\\
&\langle T,N \rangle =0\ ,\ 
\langle T,T \rangle =1\ ,\ 
\langle N,N \rangle =1,
\end{cases}
\end{equation}
where $\kappa$ is the geodesic curvature and $\langle\cdot,\cdot\rangle=g$ the standard Riemannian metric on $S^2$.  
 Then, we have the following.

\vspace{10pt}

\begin{prop}\label{8}
Let $\gamma:I\rightarrow (S^2,\langle\cdot,\cdot\rangle)$ be a smooth curve parametrized by the arc length.
Then, $\gamma $ is a $3$-harmonic curve if and only if
\begin{equation*}
\begin{cases}
\kappa^{(4)}-12(\kappa')^2-10\kappa^2\kappa''+\kappa^5-3\kappa(\kappa')^2
+\kappa''-\kappa^3=0,\\
\kappa\kappa^{(3)}-2\kappa^3\kappa'+2\kappa'\kappa''=0,
\end{cases}
\end{equation*}
where $\kappa $ is the geodesic curvature of $\gamma$.
\end{prop}


\begin{proof}
We calculate
 $(\nabla^N_{\gamma'}\nabla^N_{\gamma'})^2\tau(\gamma)$ as follows.
\begin{equation}\label{nabla22}
\begin{split}
(\nabla^N_{\gamma'}\nabla^N_{\gamma'})^2\tau(\gamma)
&=(\kappa^{(4)}-12(\kappa')^2-10\kappa^2\kappa''+\kappa^5-3\kappa(\kappa')^2)N\\
&\ \ +(-5\kappa\kappa^{(3)}+10\kappa^3\kappa'-10\kappa'\kappa'')T.
\end{split}
\end{equation}
Therefore, $\gamma$ is $3$-harmonic if and only if
\begin{equation}
\begin{split}
&(\kappa^{(4)}-12(\kappa')^2-10\kappa^2\kappa''+\kappa^5-3\kappa(\kappa')^2+\kappa''-\kappa^3)N\\
&+(-5\kappa\kappa^{(3)}+10\kappa^3\kappa'-10\kappa'\kappa'')T=0.
\end{split}
\end{equation}
So we have Proposition \ref{8}.
\end{proof}

\vspace{10pt}

\begin{prop}\label{9}
Let $\gamma:I\rightarrow (S^2,\langle\cdot,\cdot\rangle)$ be a smooth curve parametrized by the arc length.
Then, $\gamma $ is $4$-harmonic curve if and only if
\begin{equation*}
\begin{cases}
&\kappa^{(6)}-24(\kappa'')^2-24\kappa'\kappa'''-45(\kappa')^2\kappa''-46\kappa(\kappa'')^2-81\kappa\kappa'\kappa^{(3)}-21\kappa^2\kappa^{(4)}\\
&+93\kappa^3(\kappa')^2+35\kappa^4\kappa''+12\kappa^2(\kappa')^2-\kappa^7\\
&+(\kappa^{(4)}-12(\kappa')^2-10\kappa^2\kappa''+\kappa^5-3\kappa(\kappa')^2)=0,\\
&-7\kappa\kappa^{(5)}+48\kappa\kappa'\kappa''+162\kappa^2\kappa'\kappa''+35\kappa^3\kappa^{(3)}-21\kappa^5\kappa'+69\kappa(\kappa')^3\\
&-21\kappa'\kappa^{(4)}+12(\kappa')^3-35\kappa''\kappa^{(3)})=0,
\end{cases}
\end{equation*}
where $\kappa $ is the geodesic curvature of $\gamma$.
\end{prop}


\begin{proof}
We calculate 
 $(\nabla^N_{\gamma'}\nabla^N_{\gamma'})^3\tau(\gamma)$ as follows.
\begin{equation}
\begin{split}
&\hspace{-20pt}(\nabla^N_{\gamma'}\nabla^N_{\gamma'})^3\tau(\gamma)\\
=&(\kappa^{(6)}-24(\kappa'')^2-24\kappa'\kappa'''-45(\kappa')^2\kappa''-46\kappa(\kappa'')^2\\
&\ \ -81\kappa\kappa'\kappa^{(3)}-21\kappa^2\kappa^{(4)}
+93\kappa^3(\kappa')^2+35\kappa^4\kappa''+12\kappa^2(\kappa')^2-\kappa^7)N\\
+&(-7\kappa\kappa^{(5)}+48\kappa\kappa'\kappa''+162\kappa^2\kappa'\kappa''+35\kappa^3\kappa^{(3)}-21\kappa^5\kappa'+69\kappa(\kappa')^3\\
&\ \ -21\kappa'\kappa^{(4)}+12(\kappa')^3-35\kappa''\kappa^{(3)})T.
\end{split}
\end{equation}
Therefore, using (\ref{nabla22}), $\gamma$ is $4$-harmonic if and only if
\begin{equation}
\begin{split}
(&\kappa^{(6)}-24(\kappa'')^2-24\kappa'\kappa'''-45(\kappa')^2\kappa''-46\kappa(\kappa'')^2\\
&\ \ -81\kappa\kappa'\kappa^{(3)}-21\kappa^2\kappa^{(4)}
+93\kappa^3(\kappa')^2+35\kappa^4\kappa''+12\kappa^2(\kappa')^2-\kappa^7\\
&\ \ +(\kappa^{(4)}-12(\kappa')^2-10\kappa^2\kappa''+\kappa^5-3\kappa(\kappa')^2))N\\
+&(-7\kappa\kappa^{(5)}+48\kappa\kappa'\kappa''+162\kappa^2\kappa'\kappa''+35\kappa^3\kappa^{(3)}-21\kappa^5\kappa'+69\kappa(\kappa')^3\\
&\ \ -21\kappa'\kappa^{(4)}+12(\kappa')^3-35\kappa''\kappa^{(3)})T=0.
\end{split}
\end{equation}
So we have Proposition \ref{9}.
\end{proof}

\vspace{10pt}

We show the following Theorem \ref{10}.


\begin{thm}\label{10}
Every non-harmonic $k$-harmonic curve into $(S^2,\langle\cdot,\cdot\rangle)$ 
 whose geodesic curvature is constant is $2$-harmonic.
Namely, $\gamma:I\rightarrow (S^2,\langle\cdot,\cdot\rangle)$ a smooth curve parametrized by the arc length. 
If the geodesic curvature $\kappa$ of $\gamma$ is constant, every $k$-harmonic is $2$-harmonic.
\end{thm}


\begin{proof}
$\gamma$ is $2$-harmonic curve if and only if 
$$(\nabla ^N_{\gamma '}\nabla ^N_{\gamma '})\tau (\gamma)
+\tau (\gamma)
-\langle \tau (\gamma) , \gamma '\rangle \gamma '=0.$$
Now we have $\tau (\gamma)=\kappa N$
 and
 $(\nabla^N_{\gamma'}\nabla^N_{\gamma'})\tau (\gamma)
=(\kappa ''-\kappa^2)N-3\kappa \kappa'T.$
Therefore, $\gamma$ is $2$-harmonic if and only if
 $\kappa^2=1.$
Next, we consider $k$-harmonic curve. We set geodesic curvature $\kappa$ of $\gamma$ is constant.
Then, we have
\begin{align*}(\nabla^N_{\gamma'}\nabla^N_{\gamma '})^{k-2}\tau (\gamma)
&=(-1)^{k-2}\kappa ^{2(k-2)+1}N.
\end{align*}
So, $\gamma$ is $k$-harmonic curve if and only if
\begin{align*}
0
&=(-1)^{k-2}\kappa^{2(k-2)+1}(\kappa^2-1)N.
\end{align*}
So, we have Theorem \ref{10}.
\end{proof}

\vspace{10pt}

\begin{ex}[\cite{rcsmco1}]\label{ex rcsmco1}
R. Caddeo, S. Montaldo and C. Oniciuc gave following two curves are propre $2$-harmonic curves 
$\gamma:I\rightarrow S^n\subset \mathbb{R}^{n+1}$ $(cf. \cite{rcsmco1})$
 .
\begin{align}\label{rcsmco1 eg1}
\gamma(t)=\cos (\sqrt{2}t)c_1+\sin(\sqrt{2}t)c_2+c_4,
\end{align}
where $c_1$, $c_2$ and $c_4$ are constant vectors orthogonal to each other with 
$|c_1|^2=|c_2|^2=|c_4|^2=\frac{1}{2}$.
\begin{align}\label{rcsmco1 eg2}
\gamma(t)=\cos(at)c_1+\sin(at)c_2+\cos(bt)c_3+\sin(bt)c_4,
\end{align}
where $c_1$, $c_2$, $c_3$ and $c_4$ are constant vectors orthogonal to each other with 
$|c_1|^2=|c_2|^2=|c_3|^2=|c_4|^2=\frac{1}{2}$, and $a^2+b^2=2, a^2\neq b^2$.
\end{ex}

\begin{prop}\label{eg of 3,4}
The curves $(\ref{rcsmco1 eg1})$ and $(\ref{rcsmco1 eg2})$ in Example \ref{ex rcsmco1} are also $3$-harmonic and $4$-harmonic. 
\end{prop}

\begin{proof}
We can show this proposition by a direct computation. The proof is omitted.
\end{proof}

\begin{rem}
Notice that a $2$-harmonic map implies not always to be $3$-harmonic or $4$-harmonic. 
 Proposition \ref{eg of 3,4} is non-trivial.
\end{rem}



\vspace{30pt}

{\hspace{170pt}
\footnotesize
Graduate Schoole of Information Sciences.\\
\hspace{186pt}
TOHOKU University.\\
\hspace{186pt}
Aoba 6-3-09 Aramaki Aoba-ku\\
\hspace{186pt}
Sendai-shi Miyagi, 980-8579\\
\hspace{186pt}
Japan }
\vspace{15pt}

{\footnotesize
\hspace{170pt}
Current address:\\
\hspace{185pt}
Nakakuki 3-10-9 Oyama-shi Tochigi\\
\hspace{185pt}
Japan\\
\hspace{185pt}
e-mail:shun.maeta@gmail.com\\
\hspace{185pt}
e-mail:maeta@ims.is.tohoku.ac.jp}
\end{document}